\documentclass{amsart}
\usepackage{amsfonts}
\usepackage{color}

\numberwithin{equation}{section}

\newtheorem{theorem}{Theorem}[section]

\newtheorem{lem} {Lemma}[section]

\newtheorem{example}{Example}[section]

\def\N{{\mathbb N}}
\def\R{{\mathbb R}}

\usepackage[bookmarksnumbered,colorlinks, linkcolor=blue, citecolor=red, pagebackref, bookmarks, breaklinks]{hyperref}

\usepackage{amsmath,amsthm,amsfonts,amssymb}

\usepackage[left=3cm,right=3cm,top=2.9cm,bottom=2.9cm]{geometry}

\begin{document}
\title[Maximum principles for the fractional $p(x,\cdot)$-Laplacian]{ Maximum principles and moving planes method for the fractional $p(x,\cdot)$-Laplacian }
\date{}

\vspace{ -1\baselineskip}
\begingroup\small
\begin{center}
\author{\small Anouar Bahrouni, Abdelhakim Sahbani, Ariel Salort  }

\address[A. Bahrouni] {Mathematics Department,  Faculty of Sciences, University of Monastir, 5019 Monastir, Tunisia}
\email{{\tt bahrounianouar@yahoo.fr;Anouar.Bahrouni@fsm.rnu.tn}\hfill\break\indent {\it Web page:} {\tt  http://https://www.researchgate.net/profile/Anouar-Bahrouni}}

\address[A. Sahbani] {Faculty of Mathematics and Physics , University of Tunis el Manar, Tunisia}
\email{{\tt abdelhakim.sahbani@gmail.com}\hfill\break\indent {\it Web page:} {\tt  http://https://www.researchgate.net/profile/Abdelhakim-Sahbani}}

\address[A. Salort]{Departamento  de Matem\'atica, FCEN -- Universidad de Buenos Aires, and\hfill\break\indent Instituto de C\'alculo -- CONICET\hfill\break\indent $0+\infty$ building, Ciudad Universitaria (1428), Buenos Aires, Argentina.}
\email{{\tt asalort@dm.uba.ar}\hfill\break\indent {\it Web page:} {\tt  http://mate.dm.uba.ar/$\sim$asalort}}

\end{center}
\endgroup

 \maketitle
 
\begin{abstract}
 In this paper,  we investigate the monotonicity of solutions for a nonlinear equations involving the fractional Laplacian with variable exponent.
We first prove different maximum principles  involving this operator. Then we employ the direct moving planes method to obtain monotonicity of solutions to a nonlinear equations in
which the fractional laplacian with variable exponent is present. Note that, there are no results studying  the monotonicity of solutions for local or nonlocal equations with variables exponent.  Our results are new in this setting and includes a self-contained techniques.
\end{abstract}


\section{Introduction}
Fractional Sobolev spaces with variable exponent, see for instance \cite{M33}, and fractional Orlicz-Sobolev spaces, see for instance \cite{Cianchi1, FBS}, are two
distinct non-homogeneous extensions of the classical fractional Sobolev spaces, see
\cite{M26} for an introduction.  In particular, fractional Orlicz-Sobolev spaces  have drawn some attention in the very recent years due to the connection with the nonlocal nonlonear nonstandard grow operator given by the fractional $g-$Laplacian, for further information see for instance \cite{ Cianchi1, Cianchi5, bah3, bah4, Sabri, FBS, S20} and reference to these articles.

On the other hand, the study of qualitative properties of the local and nonlocal homogeneous Laplacian is nowadays huge, see
\cite{Cheng,1,2,3,chenpreprint,im,Gidas,4,Li} and references therein.  To the best of our understanding, the exploration of symmetry results in connection with fractional $g-$Laplacian problems within the theory of non-homogeneous fractional Sobolev spaces remains largely uncharted. To date, only one paper has delved into this domain, as evidenced by the work \cite{Molina}.

The literature has remained silent on many crucial properties of solutions to the fractional equations with variable exponents, such us, the symmetry of solutions. This gap persists due to the presence of variable exponents that prevents the use of change of variables technique. The primary objective of this paper is to address this deficiency  and  give some variants of the maximum principle for the fractional
$p(x,y)-$Laplacian, from where we will deduce a
symmetry of solutions in a ball.

More precisely, we are concerned with the existence of symmetry solutions for the following equation
\begin{align}\label{eq11}
\begin{cases}
\vspace{0.1cm}&(-\Delta)^{s}_{p(x,\cdot)}u(x)=u^{q(x)},\, x\in B_1(0), \\
&u(x)=0,\,x\notin B_{1}(0),
\end{cases}
\end{align}
where $B_1=\{x\in \mathbb{R}^{N}, |x|<1\}$, $0<s<1$, $p$ and $q$ are two continuous functions. The fractional $p(x,\cdot)-$Laplacian is a non-local operator defined as
\begin{equation}\label{eq1}
(-\Delta)^{s}_{p(x,\cdot)}u(x)= \text{p.v.}\,\int_{\mathbb{R}^{N}}\frac{|u(x)-u(y)|^{p(x,y)-2}(u(x)-u(y))}{|x-y|^{N+sp(x,y)}}\,dy,
\end{equation}
where $\text{p.v.}$  stands for the Cauchy principal value, $s\in (0,1)$ and $p$ is a suitable continuous function. In
order to guarantee the integral in \eqref{eq1} is well defined, we assume that $u\in C^{1,1}_{loc}(\mathbb{R}^{N}\cap
L_{sp(x,\cdot)}$, where
$$
L_{sp(x,\cdot)}=\left\{u\in L^{1}_{loc} \text{ such that } \int_{\mathbb{R}^{N}}\frac{|u(y)|^{p(x,y)-1}}{1
+|y|^{N+sp(x,y)}}\,dy<\infty \right\}.
$$ 
We will give the details of the
above fact in Lemma \ref{def}.

\medskip

Fractional Sobolev spaces with variable exponent extend the classical fractional Sobolev space. These spaces were firstly introduced by
Kaufmann et al.  in \cite{M33}, where the authors established a compact embedding theorem of these spaces into variable exponent Lebesgue spaces, and also proved an existence result for nonlocal problems involving the fractional
$p(x,\cdot)-$Laplacian.  In \cite{bah2}, Bahrouni and Radulescu obtained
some further qualitative properties of these new spaces. After that,
some studies on this context have been performed by using different
approaches, see \cite{bah1,Winkert,An.Ky,Biswas,R1,R2,M60}. In these last references, the authors established a compact embedding theorems and proved some further qualitative properties of the fractional Sobolev space with variable exponent and the fractional
$p(x,\cdot)-$Laplace operator. 

As mentioned, another potential extension of fractional Sobolev spaces is the fractional Orlicz-Sobolev space, first introduced in \cite{FBS}. This space serves as the natural setting in which to define the fractional $g-$Laplace operator. Following this development, in \cite{Cianchi1,Cianchi5,bah3,bah4,Sabri}, the authors have established foundational results such as embedding theorems and fundamental topological properties. These findings provide the necessary groundwork for the application of variational approaches.

It is worth mentioning that the literature survey on problems involving the above non-homogeneous fractional  Laplacians is almost meagre since it is still a work in progress. Specifically, it is noteworthy that there is currently no paper dedicated to the investigation of maximum principles and  moving planes for the fractional $p(x,\cdot)-$Laplacian, tools which are fundamental to 
derive symmetry, monotonicity, and non-existence of solutions. 

In recent years, there has been extensive research into the maximum principles of classical elliptic problems, as evidenced by works such as \cite{Ber1,Ber2,Gidas,Li}. However, due to the non-local nature of the classical fractional Laplacian $(-\Delta)^s$, studying these operators presents additional challenges. To address these obstacles, Chen et al. introduced a novel approach in \cite{Chenc}, aimed at deriving integral representation formulas. These formulas serve to transform given pseudo-differential equations into equivalent integral equations, facilitating the analysis of these operators. This observation ensures the potential application of the method of moving planes in integral forms, as outlined in \cite{Chenc}, and the method of moving spheres in integral forms, as detailed in \cite{Yan}, for investigating the existence of symmetric solutions to nonlinear equations incorporating nonlocal operators. These methods have been applied successfully to study equations involving nonlocal fractional operators, and a series of fruitful results have been derived in \cite{Cheng,1,2,4,5,6}. 

When the fractional Laplacian is replaced by the fractional $p-$Laplacian, the mentioned methods become ineffective due to the nonlinearity of the operator. To address these challenges,  Chen and  Li \cite{3} introduced several innovative ideas, among which a significant contribution is the introduction of the key boundary estimate lemma. This lemma plays the role of Hopf lemma in the second step of moving planes. 

In the recent article \cite{Molina}, various versions of the maximum principle for the fractional $g-$Laplacian were derived within the context of fractional non-homogeneous Sobolev spaces. By leveraging these abstract results in conjunction with the moving planes method, the authors investigated qualitative properties of solutions, including Liouville type theorems and symmetry results. However, these techniques do not directly apply to equation \eqref{eq11} due to the variable exponent, which inhibits the use of conventional change of variable techniques. This argument plays a crucial role in establishing the maximum principles and their applications.

To overcome the mentioned drawbacks, we develop some specific tools for this setting. The primary objective of this work is to bridge this gap by introducing several formulations of the maximum principle for the fractional Laplacian with variable exponents $(-\Delta)^{s}_{p(x,\cdot)}$, from where we derive  the symmetry of solutions in a ball.  

Our strategy to establish the symmetry of a solution $u$ with respect to a given hyperplane $H$ is outlined as follows. First, we consider the function $w_{\lambda}(x)=u(x_{\lambda})-u(x)$, where $x_{\lambda}$ denotes the reflection of $x$ with respect to $H$. Next, we prove that $w_{\lambda}(x)\leq 0$ in $H$, and so interchanging the roles of $x$
and $x_{\lambda}$, we conclude that $w_{\lambda}(x)=0$ in $H$. This
fact proves that $u(x)$ would be symmetric with respect to the plane
$H$. We would like to emphasize that the main tool used to prove our result is given in the formulations of the maximum principle stated in  Theorems \ref{th1b}, \ref{th2b} and \ref{le1}. 

\medskip

We provide a brief overview of our main result and discuss some of their consequential implications. 

\noindent In Theorem \ref{th1b} we prove a Strong Maximum Principle for smooth positive supersolutions $(-\Delta)^s_{p(x,\cdot)}$ in a bounded domain $\Omega\subset \R^N$: solutions must be nonnegative in $\Omega$, and they must vanish be a.e. when equal zero at some interior point of $\Omega$. Whereas in Theorem \ref{th2b} we provide for a variant of the maximum principle for anti-symmetric functions in a bounded domain.
Finally, in Theorem \ref{le1}  we obtain a boundary estimate which plays the role of the Hopf lemma in order to apply the moving planes method.

\noindent As a consequence of our result, we establish symmetry and monotonicity of bounded positive solutions to the problem in a unit ball. More precisely, in Theorem \ref{th3} we prove that bounded positive sufficiently smooth solutions to
\begin{align}\label{eq12}
\begin{cases}
        \vspace{0.1cm}&(-\Delta)^{s}_{p(x,\cdot)}u(x)=u^{q(x)}, \ \ x\in B_1(0), \\
        &u(x)=0,~~x\notin B_{1}(0),
\end{cases}
\end{align}
must be radially symmetric and monotone decreasing about the origin.

Finally,  we can also establish the symmetry of positive solutions under natural assumptions on the right hand side $f$ concerning the whole space. More precisely, in Theorem \ref{th4} we prove that bounded positive sufficiently smooth solutions to the problem
\begin{align*}
\begin{cases} \vspace{0.05cm}
&(-\Delta)^s_{p(x,\cdot)} u(x)= f(u(x)), \quad x\text{ in } \R^N.\\ \vspace{0.05cm}
&f'(t)\leq 0 \quad \text{for } t\leq 1,\\
&\lim_{|x|\to\infty} u(x) =0.
\end{cases}
\end{align*}
must be  radially symmetric around some point in $\R^N$.

\vspace{0.2cm}
The paper is organized as follows. In Section \ref{sec2}, we give some
technical lemmas which will be useful in the proof of the main
results. In Section \ref{sec3},  we introduce several formulations of the
maximum principle. Finally, in Section \ref{sec4}, we prove that the
solutions of equation \eqref{eq11} are radially, symmetric around the
origin.

\section{Technical Lemmas} \label{sec2}
In this section, we include some useful lemmas that will be used
throughout the paper. Along of the paper $s\in(0,1)$ will denote the fractional order of the operator \eqref{eq1}. We also assume the following assumptions on the exponent function:
\medskip

$\mathrm{(P_1)}$\,\, $p: \mathbb{R}^{2N}\rightarrow (2,+\infty)$ is a continuous
function  with
\begin{equation}\label{m}
1-\frac{1}{\ln(m)}<p^{-}:=\min_{(x,y)\in \mathbb{R}^{2N}}p(x,y)\leq p(x,y)\leq p^{+}:= \max_{(x,y)\in \mathbb{R}^{2N}}p(x,y),
\end{equation}
for some $0<m<1$ and $sp^+<N$, where $s\in (0,1)$.
 
\vspace{0.1cm}
 
$\mathrm{(P_2)}$\,\, There exists a continuous function $Q:
\mathbb{R}_+\rightarrow \mathbb{R_{+}}$ such that
$$
p(x,y)=Q(|x-y|) \ \ \mbox{for all} \ \ (x,y) \in \mathbb{R}^{2N},
$$
$Q$ is nondecreasing, while $t \rightarrow t^{Q(t)}$ is increasing on $\mathbb{R}_+$. 

\medskip
Assumption $\mathrm{(P_2)}$
plays a crucial role to prove the maximum principle (Theorem \ref{th2b}). More precisely,
from condition $\mathrm{(P_2)}$, we give the technical Lemma \ref{le0} which
will be the key of the proof of Theorem \ref{th2b}.
\begin{example}
Examples of functions satisfying $\mathrm{(P_1)}$ and $\mathrm{(P_2)}$ are $p(x,y)=Q(|x-y|)$ are the following.
\begin{itemize}
\item[(i)] For $t\geq 0$ and $0<m<1$,
$$
Q(t)= t\chi_{[0,1]}(t)+\left(\arctan(t)-\frac{\pi}{4}\right)\chi_{[1,\infty)}(t)+1-\frac{1}{\ln(m)}.
$$
\item[(ii)] For $t\geq 0$ and $0<m<1$,
$$
Q(t) = \frac{1}{1+e^{-t}} + \frac{1}{2}-\frac{1}{\ln(m)}.
$$
\end{itemize}
\end{example}

We state next some lemmas that are useful to prove our main results.

\begin{lem}\label{rem1}
Suppose that assumption $\mathrm{(P_1)}$ is satisfied. Then, we
have
$$
0<C_x=\displaystyle \inf_{y\in \mathbb{R}^N,|x-y|>1}\frac{|x-y|^{N+sp(x,y)}}{1+|y|^{N+sp(x,y)}}, \ \ \forall x\in \mathbb{R}^{N}\setminus \{0\}.
$$
\end{lem}
\begin{proof}
The proof is simple and we omit it.
\end{proof}
Now, as a consequence of the above lemma, we prove that the space
$L_{sp(x,\cdot)}\cap C^{1,1}_{loc}(\mathbb{R}^{N})$ is enough for the
fractional $p(x,\cdot)-$Laplacian to be well defined.
\begin{lem}\label{def}
Let $(\mathrm{P_1)}-\mathrm{(P_2)}$ be satisfied. If $u\in
C^{1,1}_{loc}(\mathbb{R}^{N})\cap L_{sp(x,\cdot)}$ at $x\in
\mathbb{R}^{N}$, then $(-\Delta)_{p(x,\cdot)}^{s}u$ is pointwisely defined  for any $x\in \R^N$.
\end{lem}
\begin{proof}
Let $0<\epsilon<1$ and $u\in C^{1,1}_{loc}(\mathbb{R}^{N})\cap
L_{sp(x,\cdot)} $. We write
$$
\displaystyle \int_{\mathbb{R}^{N}\setminus B_{\epsilon}(x)}\frac{|u(x)-u(y)|^{p(x,y)-2}(u(x)-u(y))}{|x-y|^{N+sp(x,y)}}\,dy=I_1+I_2,
$$
where
$$I_1=\int_{B_1(x)\setminus
B_{\epsilon}(x)}\frac{|u(x)-u(y)|^{p(x,y)-2}(u(x)-u(y))}{|x-y|^{N+sp(x,y)}}\,dy
$$ and 
$$
I_2=\int_{\mathbb{R}^{N}\setminus
B_{1}(x)}\frac{|u(x)-u(y)|^{p(x,y)-2}(u(x)-u(y))}{|x-y|^{N+sp(x,y)}}\,dy.
$$
Since $u\in C^{1,1}_{loc}(\mathbb{R}^{N})$, we have
$$
u(x)-u(y)=\nabla u(x)\cdot(x-y)+O(|x-y|^{2}), \ \ \mbox{as} \ \ y \rightarrow x.
$$
Thus, if $|x-y|<1$, we can deduce that
\begin{align*}
& \left| \frac{|\nabla u(x)\cdot(x-y)+O(|x-y|^{2})|^{p(x,y)-2}(\nabla
u(x)\cdot(x-y)+O(|x-y|^{2}))}{|x-y|^{N+sp(x,y)}}\right.\\
&-\left.\frac{|\nabla
u(x)\cdot(x-y)|^{p(x,y)-2}\nabla u(x)\cdot(x-y)}{|x-y|^{N+sp(x,y)}}\right|\\
&\leq C  (p(x,y)-1) |x-y|^{2-N-sp(x,y)} |x-y|^{p(x,y)-2}\\
&\leq C |x-y|^{-N+p(x,y)(1-s)},
\end{align*}
for some positive constant $C$. Note that the assumption $\mathrm{(P_2)}$ and change of variables $w=x-y$ ensure that
$\frac{|\nabla u(x)\cdot(x-y)|^{p(x,y)-2}\nabla
u(x)\cdot(x-y)}{|x-y|^{N+sp(x,y)}}$ is odd, so its integral over
$B_1(x)\setminus B_{\epsilon}(x)$ vanishes. Therefore
\begin{align*}
I_{1}&=\left|\int_{B_1(x)\setminus B_{\epsilon}(x)}\left[
\frac{|\nabla u(x)\cdot(x-y)+O(|x-y|^{2})|^{p(x,y)-2}(\nabla
u(x)\cdot(x-y)+O(|x-y|^{2}))}{|x-y|^{N+sp(x,y)}}\right.\right.\\&-
\left.\left.\frac{|\nabla u(x)\cdot(x-y)|^{p(x,y)-2}\nabla
u(x)\cdot(x-y)}{|x-y|^{N+sp(x,y)}}\,dy\right]
\right|\\
&\leq C \int_{B_1(x)\setminus B_{\epsilon}(x)}
\frac{\,dy}{|x-y|^{N-p(x,y)(1-s)}},
\end{align*}
which implies that $I_1 $ converges as $\epsilon \rightarrow
O^{+}$.

\noindent Recall that $p$ is bounded away from $2$, so
$t\rightarrow|t|^{p(x,y)-1}$ is convex. It follows that
\begin{align}\label{deq1}
I_2&\leq C \left(\int_{\mathbb{R}^{N}\setminus B_1(x)}
\frac{|u(x)|^{p(x,y)-1}}{|x-y|^{N+sp(x,y)}}\,dy+
\int_{\mathbb{R}^{N}\setminus B_1(x)}
\frac{|u(y)|^{p(x,y)-1}}{|x-y|^{N+sp(x,y)}}\,dy\right).
\end{align}
Hence, in light of \eqref{deq1} and Lemma \ref{rem1}, we infer that
\begin{align*}
I_2 &\leq C'_{x} \left(
\max\{|u(x)|^{p^{-}},|u(x)|^{p^{+}}\}\int_{\mathbb{R}^{N}\setminus
B_1(x)} \frac{\,dy}{|x-y|^{N+sp(x,y)}}
+\int_{\mathbb{R}^{N}\setminus B_1(x)}
\frac{|u(y)|^{p(x,y)-1}\,dy} {1+|y|^{N+sp(x,y)}}\right)<+\infty.
\end{align*}
This gives the result.
\end{proof}

Now, we give a technical lemma that will be used in the proof of Theorem \ref{th2b}. In fact, the proof of Lemma \ref{lem1} is similar to that in \cite{3} (see Lemma $5.1$). In our case, we need only to check that the constant $c_0$ is independent of $x$ and $y$.
\begin{lem}\label{lem1}
Suppose that condition $\mathrm{(P_1)}$ is satisfied and consider the function $f(t)=|t|^{p(x,y)-2}t$, where $x,y\in \R^N$ and $t\in \R$. Given $t_1,t_2\in \R$, there exist $\alpha$ between $t_1$ and $t_2$,  and $c_0>0$ independent of $x$ and $y$ such that
$$
|\alpha|^{p(x,y)-2}\geq c_{0}\max\{|t_{1}|^{p(x,y)-2},|t_{2}|^{p(x,y)-2}\}, \ \ \mbox{for all} \ \ x,y\in \mathbb{R}^{N}.
$$
\end{lem}

\begin{proof}
Without loss of generality, we may assume that $|t_2|\geq |t_1|$. We
distinguish two cases:
\medskip

\noindent \textbf{Case} $1$: If $|t_1|\geq \frac{|t_2|}{2}$. So,  when $t_1$
and $t_2$ have the same sign, $\alpha$ is between $t_1$ and $t_2$,
we have
$$
|\alpha|^{p(x,y)-2} \geq |t_1|^{p(x,y)-2}\geq \left(\frac{|t_2|}{2}\right)^{p(x,y)-2}\geq \frac{|t_2|^{p(x,y)-2}}{2^{p^+-2}},
$$
which gives that $c_0=\frac{1}{2^{p^+-2}}$.
Now, we suppose that $t_1$ and $t_2$ are of opposite sign. By the Mean Value Theorem, there exists $\alpha$ between $t_1$ and $t_2$ such that 
\begin{equation} \label{eqqq}
f(t_{2})-f(t_{1})=f'(\alpha)(t_{2}-t_{1}).
\end{equation}
Then using condition $\mathrm{(P_1)}$ \normalcolor, for every $x,y\in \mathbb{R}^{N}$, we get
\begin{align*}
2(p^{+}-1) |\alpha|^{p(x,y)-2}|t_2| &\geq  2(p(x,y)-1)
|\alpha|^{p(x,y)-2}|t_2|\geq
|f'(\alpha)||t_2-t_1|\\&=|f(t_{2})-f(t_{1})|\geq|f(t_2)|=|t_2|^{p(x,y)-1},
\end{align*}
which implies that
$$|\alpha|^{p(x,y)-2}\geq c_0|t_2|^{p(x,y)-2}=c_0\max\{|t_2|^{p(x,y)-2},|t_1|^{p(x,y)-2}\},$$
whith $c_0=\frac{1}{2(p^{+}-1)}$.

\medskip

\noindent \textbf{Case} $2$: If $|t_1|\leq \frac{|t_2|}{2}$. 
Using as before the Mean Value Theorem, there exists $\alpha$ between $t_1$ and $t_2$ for which \eqref{eqqq} holds. Then, in view of assumption $\mathrm{(P_1)}$, we obtain that  the following holds for   every $x,y\in \mathbb{R}^{N}$,
\begin{align*}
2(p^{+}-1) |\alpha|^{p(x,y)-2}|t_2| &\geq  2(p(x,y)-1)
|\alpha|^{p(x,y)-2}|t_2|\geq |f(t_2)-f(t_1)|\\&\geq
|f(t_{2})|-|f(t_{1})|= |t_2|^{p(x,y)-1}- |t_1|^{p(x,y)-1}\\
&\geq \left(1-\frac{1}{2^{p(x,y)-1}}\right)|t_2|^{p(x,y)-1}\geq
\left(1-\frac{1}{2^{p^--1}}\right)|t_2|^{p(x,y)-1} ,
\end{align*}
which proves that
$$
|\alpha|^{p(x,y)-2}\geq c_0|t_2|^{p(x,y)-2}, \quad \text{ with } c_0=\frac{1}{p^{+}-1}\frac{2^{p^--1}-1}{2^{p^-}}.
$$
This
completes the proof of our desired result.
\end{proof}
\section{Maximum principles } \label{sec3}
This section is devoted to the proof of several formulations of the
maximum principle.
\begin{theorem} \label{th1b}
Let $\Omega$ be a bounded domain in $\mathbb{R}^{N}$ and suppose
that $\mathrm{(P_1)}$-$\mathrm{(P_2)}$ hold. Moreover, assume that $u \in
C^{1,1}_{loc}(\mathbb{R}^{N})\cap L_{sp(x,\cdot)}$ be lower
semi-continuous on $\overline{\Omega_{}}$, and satisfies
\begin{align}\label{eq2b}
\begin{cases}
\vspace{0.1cm}
(-\Delta)^{s}_{p(x,\cdot)}u(x)\geq 0\quad &\mbox{in }\Omega, \\
u(x)\geq 0,     \quad &\mbox{in }\mathbb{R}^{N}\setminus \Omega.
\end{cases}
\end{align}
Then
\begin{equation}\label{eq3b}
u(x)\geq 0\quad \text{ for all }x\in \Omega.
\end{equation}
If $u(x)=0$ at some point $x\in \Omega$, then
$$
u(x)=0 \text{ almost everywhere in }\mathbb{R}^{N}.
$$
\end{theorem}

\begin{proof} If \eqref{eq3b} does not hold, then the lower semi-continuity of u on $\overline{\Omega}$ implies that there exists $x^{0}\in \overline{\Omega}$ such that
$$
u(x^{0})= \min_{\overline{\Omega}} u <0.
$$
Using the fact that $u(x)\geq 0$ in $\mathbb{R}^{N}\setminus
\Omega$, we get
\begin{align*}
(-\Delta)^{s}_{p(x^{0},\cdot)}u(x^{0})&= \text{p.v.}\, \int_{\mathbb{R}^{N}}\frac{|u(x^{0})-u(y)|^{p(x^{0},y)-2}(u(x^{0})-u(y))}{|x^{0}-y|^{N+sp(x^{0},y)}}\,dy\\
&\leq  \int_{\mathbb{R}^{N}\setminus \Omega}\frac{|u(x^{0})-u(y)|^{p(x^{0},y)-2}(u(x^{0})-u(y))}{|x^{0}-y|^{N+sp(x^{0},y)}}\,dy< 0.
\end{align*}
This contradicts the first inequality in \eqref{eq2b}.  Therefore,
we conclude the assertion \eqref{eq3b}. If at some point $x^{0}\in
\Omega$, $u(x^{0})=0$, then
\begin{eqnarray}\label{eq4b}
(-\Delta)^{s}_{p(x^0,\cdot)}u(x^{0})=\text{p.v.}\, \int_{\mathbb{R}^{N}}\frac{|u(y)|^{p(x^{0},y)-2}(-u(y))}{|x^{0}-y|^{N+sp(x^{0},y)}}\,dy
\leq 0.
\end{eqnarray}
Next, from the first inequality in \eqref{eq2b}, we have
 \begin{eqnarray}\label{eq5b}
\text{p.v.}\, \int_{\mathbb{R}^{N}}\frac{|u(y)|^{p(x^{0},y)-2}(-u(y))}{|x^{0}-y|^{N+sp(x^{0},y)}}\,dy\geq
0.
\end{eqnarray}
Combining \eqref{eq4b} and \eqref{eq5b}, we obtain
\begin{eqnarray*}
\text{p.v.}\, \int_{\mathbb{R}^{N}}\frac{|u(y)|^{p(x^{0},y)-2}u(y)}{|x^{0}-y|^{N+sp(x^{0},y)}}\,dy=
0.
\end{eqnarray*}
Then $u(x)= 0$ a.e. in $\R^N$, which concludes the proof.
\end{proof}

In order to state our main result we introduce some notation.

\noindent Let $\lambda\in \mathbb{R}$ and
$$
T_{\lambda}=\{x\in \mathbb{R}^{N}, \ x_{1}=\lambda\},
$$
be the hyperplane at height $\lambda$. Let
$$
H_{\lambda}=\{x\in \mathbb{R}^{N}, \ x_{1}<\lambda \},
$$
be the upper half-space. For each $x\in H_{\lambda}$ let
$$
x_{\lambda}=(2\lambda -x_{1}, x_{2},...,x_{N})
$$
be the reflection of $x$ about the plane $T_{\lambda}$. We denote
$$
w_{\lambda}(x)=u(x_{\lambda})- u(x)=u_{\lambda}(x)- u(x).
$$
It is easy to remark that 
$$
w_{\lambda}(x_{\lambda})=-
w_{\lambda}(x).
$$ 
Next, we give a variant of maximum principle for
anti-symmetric functions in bounded domain.

\begin{theorem} \label{th2b}
Let $\Omega_{\lambda}=H_{\lambda}\cap B(0,1)$ be a bounded domain in $H_{\lambda}$. Assume that
assumptions  $\mathrm{(P_1)}$-$\mathrm{(P_2)}$ hold. Let $u \in
C^{1,1}_{loc}(\mathbb{R}^{N})\cap L_{sp(x,\cdot)}$  be a lower semi-continuous function on $\overline{B}_1(0)$ such that
\begin{equation}\label{eq6b}
\begin{cases}
\vspace{0.1cm}
&(-\Delta)^{s}_{p(x,\cdot)}u_{\lambda}(x)-(-\Delta)^{s}_{p(x,\cdot)}u(x) \geq 0\quad \mbox{in }B_1(0), \\
\vspace{0.1cm}
&w_{\lambda}(x)\geq 0,     \quad \mbox{in }H_{\lambda}\setminus B_1(0).\\               
&u\in (0,m),~~ 0<m<1,
\end{cases}        
\end{equation}
Then
\begin{equation}\label{eq7b}
w_{\lambda}(x)\geq 0,~~~~~~\forall~~ x\in H_{\lambda}.
\end{equation}
If $w_{\lambda}(x)=0$ at some point $x\in B_1(0)$, then
$$
w_{\lambda}(x)=0 ~~almost~~everywhere~~in~~\mathbb{R}^{N}.
$$
\end{theorem}

Before giving the proof of the above theorem, we need to establish
the following crucial lemma.
\begin{lem}\label{le0} Let $\Omega$ be a bounded domain in $H_{\lambda}$. Suppose that  conditions $\mathrm{(P_1)}$-$\mathrm{(P_2)}$ are satisfied. Then, for $x^{0}\in \overline{\Omega}$, we have that
$$
\frac{1}{|x^{0}-y|^{N+sp(x^{0},y)}}>\frac{1}{|x^{0}-y_{\lambda}|^{N+sp(x^{0},y_{\lambda})}} \quad \text{for all }y \in H_{\lambda}.
$$
\end{lem}

\begin{proof}
Firstly, we observe that, for $x^{0}\in \overline{\Omega}$,
$$|x^{0}-y|<|x^{0}-y_{\lambda}|, \forall~~y~~\in H_{\lambda}. $$
Observe that, due to $\mathrm{(P_2)}$, the function $t\rightarrow t^{Q(t)}$ is increasing in $(0,+\infty)$. Therefore
\begin{align*}
|x^{0}-y|^{N+sQ(|x^{0}-y|)} <
|x^{0}-y_{\lambda}|^{N+sQ(|x^{0}-y_{\lambda}|)}\quad \forall y \in H_{\lambda},
\end{align*}
 which implies that
$$
\frac{1}{|x^{0}-y|^{N+sp(x^{0},y)}}>\frac{1}{|x^{0}-y_{\lambda}|^{N+sp(x^{0},y_{\lambda})}}\quad \forall y \in H_{\lambda}.
$$
That finishes the proof.
\end{proof}
Next we give the proof of  Theorem \ref{th2b}.

\begin{proof} If \eqref{eq7b} does not hold, then the lower semi-continuity of u on $\overline{B}_1(0)$ implies that there exists $x^{0}\in \overline{B}_1(0)$ such that
$$
w_{\lambda}(x^{0})= \min_{\overline{B}_1(0) }w_{\lambda} <0.
$$
We denote $$f=f_{p(x,y)}(t)=|t|^{p(x,y)-2}t, \ f_{*}=f_{p(x,y_{\lambda})}(t)=|t|^{p(x,y_{\lambda})-2}t \  \mbox{and} \
 g(t)=f(t)-f_{*}(t).$$
 Note that $f$ and $f_{*}$ are two increasing functions. Let
$$
\kappa(x^{0},y)=\frac{1}{|x^{0}-y|^{N+sp(x^{0},y)}}-\frac{1}{|x^{0}-y_{\lambda}|^{N+sp(x^{0},y_{\lambda})}}
$$
and
\begin{eqnarray*}
\Gamma(x^{0})=(-\Delta)^{s}_{p(x^0,\cdot)}u_{\lambda}(x^{0})-(-\Delta)^{s}_{p(x^0,\cdot)}u(x^{0}).
\end{eqnarray*}
We write
\begin{align*}
\Gamma(x^{0})&= \int_{\mathbb{R}^{N}}\frac{f(u_{\lambda}(x^{0})-u_{\lambda}(y))-f(u(x^{0})-u({y}))}{|x^{0}-y|^{N+sp(x^{0},y)}}\,dy\nonumber\\
&=\int_{H_{\lambda}}\frac{f(u_{\lambda}(x^{0})-u_{\lambda}(y))-f(u(x^{0})-u({y}))}{|x^{0}-y|^{N+sp(x^{0},y)}}\,dy\nonumber\\
&\quad +\int_{H_{\lambda}}\frac{f_{*}(u_{\lambda}(x^{0})-u(y))-f_{*}(u(x^{0})-u_{\lambda}(y))}{|x^{0}-y_{\lambda}|^{N+sp(x^{0},y_{\lambda})}}\,dy\nonumber\\
&= \int_{H_{\lambda}}\kappa(x^{0},y) \Big(f(u_{\lambda}(x^{0})-u_{\lambda}(y))-f(u(x^{0})-u({y}))\Big)\,dy\nonumber\\
&\quad +\int_{H_{\lambda}}\frac{f(u_{\lambda}(x^{0})-u_{\lambda}(y))-f(u(x^{0})-u(y))+f_{*}(u_{\lambda}(x^{0})-u(y))-f_{*}(u(x^{0})-u_{\lambda}(y))}
{|x^{0}-y_{\lambda}|^{N+sp(x^{0},y_{\lambda})}}\,dy\nonumber\\
&:= J_{1}(x^{0})+  J_{2}(x^{0}).
\end{align*}
Observe that
\begin{eqnarray}\label{eq8b}
[u_{\lambda}(x^{0})-u_{\lambda}(y)]-[u(x^{0})-u(y)]=w_{\lambda}(x^{0})-w_{\lambda}(y)\leq 0.
\end{eqnarray}
Therefore, using the fact that  $w_{\lambda}(y_{\lambda})=-
w_{\lambda}(y)$ and $w_{\lambda}(x^{0})<0$, we deduce that
$w_{\lambda}(x^{0})-w_{\lambda}(y)$ is not identically null. Then, using
\eqref{eq8b} and Lemma \ref{le0}, we conclude that
\begin{eqnarray}\label{eq9b}
 J_{1}(x^{0})<0.
 \end{eqnarray}
 On the other hand, one has
\begin{align*}
J_2(x^{0})&=\int_{H_{\lambda}}\frac{f(u_{\lambda}(x^{0})-u_{\lambda}(y))-f(u(x^{0})-u(y))+f_{*}(u_{\lambda}(x^{0})-u(y))-f_{*}(u(x^{0})-u_{\lambda}(y))}
{|x^{0}-y_{\lambda}|^{N+sp(x^{0},y_{\lambda})}}\,dy\\
&=\int_{H_{\lambda}}\frac{\Big(f(u_{\lambda}(x^{0})-u_{\lambda}(y)) -f_{*}(u_{\lambda}(x^{0})-u_{\lambda}(y))\Big)
 +f_{*}(u_{\lambda}(x^{0})-u_{\lambda}(y))-f_{*}(u(x^{0})-u_{\lambda}(y))}{|x^{0}-y_{\lambda}|^{N+sp(x^{0},y_{\lambda})}}\,dy\nonumber \\
&\quad-\int_{H_{\lambda}}\frac{\Big(f(u(x^{0})-u(y))-f_{*}(u(x^{0})-u(y))\Big)-f_{*}(u_{\lambda}(x^{0})-u(y))+f_{*}(u(x^{0})-u(y))}
{|x^{0}-y_{\lambda}|^{N+sp(x^{0},y_{\lambda})}}\,dy,
\end{align*}
which can be written as
\begin{align*}
J_2(x^{0})
&=\int_{H_{\lambda}}\frac{\Big(f(u_{\lambda}(x^{0})-u_{\lambda}(y)) -f_{*}(u_{\lambda}(x^{0})-u_{\lambda}(y))\Big)-
\Big(f(u(x^{0})-u(y))-f_{*}(u(x^{0})-u(y))\Big)}{|x^{0}-y_{\lambda}|^{N+sp(x^{0},y_{\lambda})}}\,dy \\
 &\quad+\int_{H_{\lambda}}\frac{f_{*}(u_{\lambda}(x^{0})-u_{\lambda}(y))-f_{*}(u(x^{0})-u_{\lambda}(y))  +f_{*}(u_{\lambda}(x^{0})-u(y))-f_{*}(u(x^{0})-u(y))}{|x^{0}-y_{\lambda}|^{N+sp(x^{0},y_{\lambda})}}\,dy\\
 &=\int_{H_{\lambda}}\frac{g\Big(u_{\lambda}(x^{0})-u_{\lambda}(y)\Big)-
g\Big( u(x^{0})-u(y))\Big)}{|x^{0}-y_{\lambda}|^{N+sp(x^{0},y_{\lambda})}}\,dy\nonumber \\
 &\quad+\int_{H_{\lambda}}\frac{f_{*}(u_{\lambda}(x^{0})-u_{\lambda}(y))-f_{*}(u(x^{0})-u_{\lambda}(y))  +f_{*}(u_{\lambda}(x^{0})-u(y))-f_{*}(u(x^{0})-u(y))}{|x^{0}-y_{\lambda}|^{N+sp(x^{0},y_{\lambda})}}\,dy.
 \end{align*}
 Thus, there exist $t_0\in\left( u_{\lambda}(x^{0})-u_{\lambda}(y), u(x^{0})-u(y)\right)$, $t_1\in \left( u_{\lambda}(x^{0})-u_{\lambda}(y), u(x^{0})-u_{\lambda}(y)\right)$ and $t_2\in \left( u_{\lambda}(x^{0})-u(y), u(x^{0})-u(y)\right)$ such that
 \begin{equation}\label{s}
  J_2(x^{0})=\int_{H_{\lambda}}\frac{\Big(w_{\lambda}(x^{0})-w_{\lambda}(y)\Big) g'(t_{0})}
{|x^{0}-y_{\lambda}|^{N+sp(x^{0},y_{\lambda})}}\,dy+ w_{\lambda}(x^{0})
\int_{H_{\lambda}}\frac{f_{*}'(t_{1})+f_{*}'(t_{2}) }{|x^{0}-y_{\lambda}|^{N+sp(x^{0},y_{\lambda})}}\,dy,
  \end{equation}
which implies, since $ w_{\lambda}(x^{0})<0$ and $f_{*}'(t)\geq 0$, that
  \begin{eqnarray}\label{hhh}
   J_2(x^{0})\leq\int_{H_{\lambda}}\frac{  \Big(w_{\lambda}(x^{0})-w_{\lambda}(y)\Big) g'(t_{0})}
{|x^{0}-y_{\lambda}|^{N+sp(x^{0},y_{\lambda})}}\,dy.
\end{eqnarray}
In what follows we  show that $g'(t_{0})\geq 0$ with 
\begin{equation}\label{g}g'(t_{0})=(p(x,y)-1)|t_{0}|^{p(x,y)-2}-(p(x,y_{\lambda})-1)|t_{0}|^{p(x,y_{\lambda})-2}.
\end{equation}
For this, we consider the function $h(t)=(t-1)|t_{0}|^{t-2}$ for $t>0$. By a simple computation, we obtain
$$
 h'(t)=|t_{0}|^{t-2}\Big(1+(t-1)\ln(|t_{0}|)\Big).
$$
Recalling that $u\in (0,m)$ and $m<1$, then $|t_{0}|\in (0,m).$ This shows that $h'(t)<0$, for all $t\geq 1-\frac{1}{\ln(|t_{0}|)}$. Consequently, $h$ is decreasing for all $t\geq 1-\frac{1}{\ln(m)}$. On the other hand, in light of $\mathrm{(P_2)}$, we have 
$$
1-\frac{1}{\ln(m)}\leq p(x,y)\leq p(x,y_{\lambda}).
$$ 
Then $h(p(x,y))\geq h(p(x,y_{\lambda}))$. This proves, due to \eqref{g}, that
\begin{eqnarray}\label{hh}
 g'(t_{0})=(p(x,y)-1)|t_{0}|^{p(x,y)-2}-(p(x,y_{\lambda})-1)|t_{0}|^{p(x,y_{\lambda})-2}\geq 0,
\end{eqnarray}
and so, using the fact $w_{\lambda}(x^{0})-w_{\lambda}(y)\leq0$, \eqref{hhh} and \eqref{hh}, we infer that
\begin{eqnarray}\label{eq103b}
 J_{2}(x^{0})\leq 0.
\end{eqnarray}
 Consequently, by combining \eqref{eq9b} and \eqref{eq103b}, we obtain
\begin{equation} \label{ineqq}
(-\Delta)^{s}_{p(x^0,\cdot)}u_{\lambda}(x^{0})-(-\Delta)^{s}_{p(x^0,\cdot)}u(x^{0})<0.
\end{equation}
 This contradicts the first equation of \eqref{eq6b}. Hence we must have
 \begin{equation*}
w_{\lambda}(x)\geq 0,~~~~~~\forall x\in B_1(0).
\end{equation*}
Finally, if $w_{\lambda}(x^{0})=0$ at some $x^{0}\in B_1(0)$, then $x^{0}$
is a minimum of $w_{\lambda}$ in $B_1(0)$. Thus, by \eqref{hhh}
\begin{equation}\label{h1}
 J_{2}(x^{0})\leq\int_{H_{\lambda}}\frac{-w_{\lambda}(y)g'(t_{3})}
{|x^{0}-y_{\lambda}|^{N+sp(x^{0},y_{\lambda})}}\,dy,
 \end{equation}
 and then from \eqref{g},  $J_{2}(x^{0})\leq 0$. Moreover, from \eqref{eq6b}, that $J_{1}(x^{0})\geq 0$. Then
\begin{equation}\label{eq11b}
 f(u_{\lambda}(x)-u_{\lambda}(y))-f(u(x)-u({y}))\geq 0.
\end{equation}
This proves, using the increasing property of $f$, 
that
$$ [u_{\lambda}(x^{0})-u_{\lambda}(y)]-[u(x^{0})-u(y)]=-w_{\lambda}(y)\geq 0, \ \ \forall y\in H_{\lambda}.
$$
Hence we must have $w_{\lambda}(y)= 0,~~~~~~\forall~~ y\in H_{\lambda}$.

\noindent Consequently, since $w_{\lambda}(y_{\lambda})=- w_{\lambda}(y)$,
one has $w_{\lambda}(y)= 0,~~~~~~\forall~~ y\in \mathbb{R}^{N}$ as desired.
\end{proof}
We prove the following key boundary estimate lemma which will play
the role of the Hopf lemma in the second step of moving planes stated in Theorem
\ref{th3}.

\begin{theorem}\label{le1} Assume that conditions $\mathrm{(P_1)}$-$\mathrm{(P_2)}$ hold. Let $\{\lambda_{k}\}_{k\in\N}$ be a real sequence and $x^{k}\in H_{\lambda_{k}}$ such that
$\lambda_{k}\rightarrow \lambda_{0}$,
\begin{equation*}
  w_{\lambda_{k}}(x^{k})=\min_{H_{\lambda_{k}}}w_{\lambda_{k}}\leq 0,~~and~~x^{k}\rightarrow x^{0}\in T_{\lambda_{0}}.
\end{equation*}
We suppose, for $x\in H_{\lambda_{0}}$, that $w_{\lambda_{0}}(x)>0$.
Let
$$\delta_{k}=dist(x^{k}, T_{\lambda_{k}})\equiv
|\lambda_{k}- x_{1}^{k}|.
$$ Then, 
\begin{equation*}
 \overline{\lim_{\delta_{k}\rightarrow 0}}\frac{1}{\delta_{k}}\{(-\Delta)^{s}_{p(x^{k},\cdot)}u_{\lambda_{k}}(x^{k})-
 (-\Delta)^{s}_{p(x^{k},\cdot)}u(x^{k})\}<0.
\end{equation*}
\end{theorem}

\begin{proof}
Let
\begin{eqnarray*}
\Gamma_k=\frac{1}{\delta_{k}}\{(-\Delta)^{s}_{p(x^{k},\cdot)}u_{\lambda_{k}}(x^{k})-(-\Delta)^{s}_{p(x^{k},\cdot)}u(x^{k})\}.
\end{eqnarray*}
Proceeding as in the proof of Theorem \ref{th2b}, we obtain
\begin{align*}
\Gamma_k
&= \frac{1}{\delta_{k}} \int_{H_{\lambda_{k}}}\kappa(x^k,y) \left(f(u_{\lambda_{k}}(x^{k})-u_{\lambda_{k}}(y))-f(u(x^{k})-u({y}))\right)\,dy\\
&\quad +\frac{1}{\delta_{k}}\int_{H_{\lambda_{k}}}\frac{f(u_{\lambda_{k}}(x^{k})-u_{\lambda_{k}}(y))-f(u(x^{k})-u(y))+f_{*}(u_{\lambda_{k}}(x^{k})-u(y))-f_{*}(u(x^{k})-u_{\lambda_{k}}(y))}
{|x^{k}-y_{\lambda_{k}}|^{N+sp(x^{k},y_{\lambda_{k}})}}\,dy\nonumber\\
&:= J_{1}(x^{k})+  J_{2}(x^{k}),
\end{align*}
where we have denoted
$$
\kappa(x^k,y)=\frac{1}{|x^k-y|^{N+sp(x^k,y)}}-\frac{1}{|x^k-y_{\lambda_k}|^{N+sp(x^k,y_{\lambda_k})}}.
$$
Similarly to \eqref{eq103b}, we show that 
\begin{eqnarray}\label{eqtss}
J_{2}(x^{k}) \leq 0.
\end{eqnarray}
Next, we estimate $J_{1k}$. To this end, note that
$$
\kappa(x^k,y)\to \kappa(x^0,y) \quad \text{ as } k\to\infty.
$$
Exploiting the fact that $\kappa(x^0,y)>0$ in $H_{\lambda_0}$ due to Lemma \ref{le0} and that 
$$
[u_{\lambda_{0}}(x^{0})-u_{\lambda_{0}}(y)]-[u(x^{0})-u({y})]=w_{\lambda_{0}}(x^{0})-w_{\lambda_{0}}(y)<0,
$$
together with  the monotonicity of $f$, we deduce that for all $y\in H_{\lambda_{0}}$ it holds that
\begin{align*}
f(u_{\lambda_{k}}(x^{k})-u_{\lambda_{0}}(y))-f(u(x^{k})-u({y}))&\rightarrow
f(u_{\lambda_{0}}(x^{0})-u_{\lambda_{0}}(y))-f(u(x^{0})-u({y}))
<0,
\end{align*}
as $k\rightarrow \infty$ . So, from the above pieces of information, we infer that
\begin{equation}\label{eqsss}
\overline{\lim_{\delta_{k}\rightarrow 0}}J_{1}(x^{k})<0.
\end{equation}
Consequently, combining \eqref{eqtss} and \eqref{eqsss},
we conclude that
\begin{equation*}
 \overline{\lim_{\delta_{k}\rightarrow 0}}\frac{1}{\delta_{k}}\{(-\Delta)^{s}_{p(x^{k},\cdot)}u_{\lambda_{k}}(x^{k})-
 (-\Delta)^{s}_{p(x^{k},\cdot)}u(x^{k})\}<0.
\end{equation*}
This ends the proof.
\end{proof}

\section{Radially symmetric of solutions to a fractional $p(x,\cdot)$-Laplacian problem} \label{sec4}

In this section, we work under the conditions introduced in Theorem
\ref{th2b}. More precisely, we study the following nonlinear system
involving fractional $p(x,\cdot)$-Laplacian in the unit ball:
\begin{align}\label{eq12}
\begin{cases}
        \vspace{0.1cm}&(-\Delta)^{s}_{p(x,\cdot)}u(x)=u^{q(x)}, \ \ x\in B_1(0), \\
        &u(x)=0,~~x\notin B_{1}(0),
\end{cases}
\end{align}

\begin{theorem}	\label{th3}
Suppose that $(\mathrm{P_{1})}$-$\mathrm{(P_{2})}$ hold. Moreover,
assume that $u\in C^{1,1}_{loc}\cap L_{sp(x,\cdot)}$ and $u\in (0, 1)$ is a solution of \eqref{eq12} with $q\in C(B_1(0), (1,\infty))$. Then $u$ must be radially symmetric and monotone decreasing about the origin.
\end{theorem}

\begin{proof}
Let $H_{\lambda}$, $u_{\lambda}(x)$ and $w_{\lambda}(x)$ be defined as in Section $3$. Denote 
$$
\Omega_{\lambda}=H_{\lambda}\cap B_{1}(0).
$$
Then in $\Omega_{\lambda}$, we have
\begin{equation}\label{eq2c}
(-\Delta)^{s}_{p(x, \cdot)}u_{\lambda}(x)- (-\Delta)^{s}_{p(x, \cdot)}u(x)=q(x)\xi^{q(x)-1}_{\lambda}(x)w_{\lambda}(x),
\end{equation}
where $\xi_{\lambda}(x)$ is a value between $u_{\lambda}(x)$ and $u(x)$.
Thus at any point $x\in \Omega_{\lambda}$ where $w_{\lambda}(x)\leq 0$, we have
\begin{equation}\label{eq31}
 (-\Delta)^{s}_{p(x, \cdot)}u_{\lambda}(x)- (-\Delta)^{s}_{p(x, \cdot)}u(x) \geq q(x) u^{q(x)-1}(x)w_{\lambda}(x).
\end{equation}
We divide the proof into two steps.

\medskip

\noindent \textbf{Step 1}. In this step, we show that for $\lambda$ sufficiently closed to -1, it holds
\begin{eqnarray}\label{eq4c}
w_{\lambda}(x)\geq 0, \forall~~x\in \Omega_{\lambda}.
\end{eqnarray}
Otherwise, there exists $x^{0}\in \Omega_{\lambda}$, such that
$$
w_{\lambda}(x^{0})=\min_{\Omega_{\lambda}} w_{\lambda}=\min_{H_{\lambda}} w_{\lambda}<0.
$$
Denote
$$
\Gamma(x^{0})=(-\Delta)^{s}_{p(x^{0}, \cdot)}u_{\lambda}(x^{0})-(-\Delta)^{s}_{p(x^{0}, \cdot)}u(x^{0}).
$$
Then, similar to the proof of Theorem \ref{th2b}, we obtain
\begin{align*}
\Gamma(x^{0})&=\int_{H_{\lambda}}\kappa(x^{0},y) \Big(f(u_{\lambda}(x^{0})-u_{\lambda}(y))-f(u(x^{0})-u({y}))\Big)\,dy\nonumber\\
&\quad+\int_{H_{\lambda}}\frac{f(u_{\lambda}(x^{0})-u_{\lambda}(y))-f(u(x^{0})-u(y))+f_{*}(u_{\lambda}(x^{0})-u(y))-f_{*}(u(x^{0})-u_{\lambda}(y))}
{|x^{0}-y_{\lambda}|^{N+sp(x^{0},y_{\lambda})}}\,dy\nonumber\\
& :=   J_{1}(x^{0})+  J_{2}(x^{0}),
\end{align*}
Proceeding as in \eqref{eq9b}, we infer that
\begin{eqnarray}\label{eq33}
 J_{1}(x^{0})<0.
 \end{eqnarray}
Next, by \eqref{s}, we get
\begin{align*}
J_{2}(x^{0})=\int_{H_{\lambda}}\frac{  \Big(w_{\lambda}(x^{0})-w_{\lambda}(y)\Big) g'(t_{0})}
{|x^{0}-y_{\lambda}|^{N+sp(x^{0},y_{\lambda})}}\,dy + w_{\lambda}(x^{0})
\int_{H_{\lambda}}\frac{f_{*}'(t_{1})+f_{*}'(t_{2}) }
{|x^{0}-y_{\lambda}|^{N+sp(x^{0},y_{\lambda})}}\,dy,
  \end{align*}
and 
 \begin{align}\label{eq34}
 (-\Delta)^{s}_{p(x^0,\cdot)}u_{\lambda}(x^{0})-(-\Delta)^{s}_{p(x^0,\cdot)}u(x^{0})\leq  w_{\lambda}(x^{0}) J_{21},
\end{align}
where $$
J_{21}=\int_{H_{\lambda}}\frac{f_{*}'(t_{1})+f_{*}'(t_{2}) }
{|x^{0}-y_{\lambda}|^{N+sp(x^{0},y_{\lambda})}}\,dy.
$$
Now, we estimate $J_{21}$ to derive a contradiction where $\lambda$ is sufficiently close to -1.
Let $D=H_{\lambda}\backslash\Omega_{\lambda}$. Noting that $u(y)=0$ in $D$ and exploiting Lemma \ref{lem1}, we deduce that
\begin{align}\label{eq6c}
  J_{21}&\geq \int_{H_{\lambda}}\frac{(p(x^{0},y_{\lambda})-1)|t_2|^{p(x^{0},y_{\lambda})-2}}{|x^{0}-y_{\lambda}|^{N+sp(x^{0},y_{\lambda})}}\,dy\nonumber\\
&\geq\int_{H_{\lambda}}\frac{(p(x^{0},y_{\lambda})-1)c_{0}|u(x^{0})-u(y)|^{p(x^{0},y_{\lambda})-2}}{|x^{0}-y_{\lambda}|^{N+sp(x^{0},y_{\lambda})}}\,dy\nonumber\\
&\geq c_{0}(p^--1)\int_{D}\frac{|u(x^{0})|^{p(x^{0},y_{\lambda})-2}}{|x^{0}-y_{\lambda}|^{N+sp(x^{0},y_{\lambda})}}\,dy\nonumber\\
&\geq c_{0}(p^--1) \int_{\Omega_{\lambda+1}\setminus
\Omega_{\lambda}}\frac{|u(x^{0})|^{p(x^{0},y_{\lambda})-2}}{|x^{0}-y_{\lambda}|^{N+sp(x^{0},y_{\lambda})}}\,dy\nonumber
\\ &\geq C\frac{\min\{|u(x^{0})|^{p^+-2},|u(x^{0})|^{p^--2}\}}{\delta^{sp^-}},
\end{align}
where $\delta=\lambda+1$ is the width of the region $\Omega_{\lambda}$ in the $x_{1}$-direction. Hence, by \eqref{eq34} and \eqref{eq6c}, we obtain
\begin{align*}
&(-\Delta)^{s}_{p(x^{0}, \cdot)}u_{\lambda}(x^{0})-(-\Delta)^{s}_{p(x^{0}, \cdot)}u(x^{0})-q(x^0)u^{q(x^{0})-1}(x^{0})w_{\lambda}(x^{0})\\ & \leq  C w_{\lambda}(x^{0})\frac{\min(|u(x^{0})|^{p^+-2},|u(x^{0})|^{p^--2})}{ \delta^{sp^-}} - q(x^{0})u^{q(x^{0})-1}(x^{0})w_{\lambda}(x^{0}).
\end{align*}
Therefore, since  $w_{\lambda}(x^{0})<0$, we can observe that
\begin{align*}
&(-\Delta)^{s}_{p(x^{0}, \cdot)}u_{\lambda}(x^{0})-(-\Delta)^{s}_{p(x^{0}, \cdot)}u(x^{0})-q(x^{0})u^{q(x^{0})-1} (x^{0})w_{\lambda}(x^{0})\\
& \leq   w_{\lambda}(x^{0})\Big[C \frac{\min(|u(x^{0})|^{p^+-2},|u(x^{0})|^{p^--2})}{ \delta^{sp^-}}- q(x^{0})u^{q(x^{0})-1}(x^{0})\Big]<0,
\end{align*}
when $\delta$ is sufficiently small.
 This contradicts \eqref{eq31}, then \begin{eqnarray*}
 w_{\lambda}(x)\geq 0, \forall~~x\in \Omega_{\lambda},
\end{eqnarray*}
for $\lambda$ sufficiently close to -1.

\medskip

\noindent \textbf{Step 2:}\\
 Inequality \eqref{eq4c} provides a starting point to move the plane $T_{\lambda}$ to the right as long as \eqref{eq4c} holds to its limiting position. More precisely,
 denote
$$
\lambda_{0}=\sup\{\lambda\leq 0 \text{ such that } w_{\mu}(x)\geq 0,~~ x\in \Omega_{\mu},~~ \mu\leq \lambda\}.
$$
Next, we want to show that 
$$
\lambda_0=0.
$$
We argue by contradiction and suppose that $\lambda_0<0$. Then,
using \eqref{eq2c} and the strong maximum principle stated in Theorem \ref{th2b}, we deduce
that
\begin{equation}\label{eq14c}
w_{\lambda_{0}}(x)> 0,~~x\in \Omega_{\lambda_{0}},
\end{equation}
On the other hand, by the definition of supremum, there exists a
sequence $\lambda_{k}$  such that
$$\lambda_k\leq \lambda_{k-1}, \ \ \lambda_k\leq 0 \ \ \mbox{and} \ \ \lambda_k\rightarrow \lambda_0.$$
Again, by the definition of supremum, there exist $x_k\in
\Omega_{\lambda_{k}}$ such that
\begin{eqnarray}\label{eq15c}
  w_{\lambda_{k}}(x^{k})=\min_{\Omega_{\lambda_{k}}}w_{\lambda_{k}}<0,~~and~~\nabla
  w_{\lambda_{k}}(x^{k})=0.
\end{eqnarray}
 We may assume further, up to a subsequence, that
$$
x^k\rightarrow x^0, \ \  w_{\lambda_{0}}(x^0)\leq 0,$$
which owing to \eqref{eq14c}  implies $x^0
\in T_{\lambda_{0}}.$
 It follows, in view of  \eqref{eq15c}, the continuity of
$w_{\lambda}$ and its derivative with respect to both
x and $\lambda$, that
\begin{equation}\label{eq17c}
   \nabla w_{\lambda_{0}}(x^{0})=0.
\end{equation}
Further, setting
$\delta_k:=\mbox{dist}(x^k,T_{\lambda_{k}})=|\lambda_k-x_{1}^{k}|=|x^k-z^k|$
for some $z^k\in T_{\lambda_{k}}$. Recall that, from \eqref{eq2c},
we have
\begin{equation}\label{eqan1}
\frac{1}{\delta_{k}}\{(-\Delta)^{s}_{p(x^{k},\cdot)}u_{\lambda_{k}}(x^{k})-(-\Delta)^{s}_{p(x^{k},\cdot)}u(x^{k})\}=\frac{q(x^k)}{\delta_{k}}\xi^{q(x^k)-1}_{\lambda_k}(x)w_{\lambda_k}(x^k).
\end{equation}
Notice that $w_{\lambda_{k}}=0$ on
$T_{\lambda_{k}}$, so
\begin{align}\label{eqan2}
\begin{split}
\displaystyle \lim_{k\rightarrow
\infty}\frac{w_{\lambda_{k}}(x^{k})}{|x^k-z^k|}&=\lim_{k\rightarrow
\infty}\frac{w_{\lambda_{k}}(x^{k})-w_{\lambda_{k}}(z^{k})}{|x^k-z^k|}\\&=\lim_{k\rightarrow
\infty}\frac{\nabla
w_{\lambda_{k}}(x^k)\cdot(z^k-x^k)+o(|z^k-x^k|)}{|z^k-x^k|}=0.
\end{split}
\end{align} Consequently, combining \eqref{eqan1} and
\eqref{eqan2}, we obtain
\begin{equation}\label{eq18c}
 \lim_{\delta_{k}\rightarrow 0}\frac{1}{\delta_{k}}\{(-\Delta)^{s}_{p(x^{k},\cdot)}u_{\lambda_{k}}(x^{k})-(-\Delta)^{s}_{p(x^{k},\cdot)}u(x^{k})\}=0.
\end{equation}
In the last lines, we used that $q(x^k)\xi^{q(x^k)-1}_{\lambda_k}(x)w_{\lambda_k}(x^k)$ is a
bounded sequence. 
Therefore, \eqref{eq18c} gives a contradiction with
Theorem \ref{le1}. Therefore $\lambda_0=0$.\\ Since $x_{1}$ direction
can be chosen arbitrarily, we conclude that $u$ and $v$ must be
radially symmetry and monotone decreasing about the $B_{1}(0)$.
\end{proof}

Finally, using  the previous results, we can establish the symmetry of positive
solutions under natural assumptions on the right hand side $f$ concerning the whole space.
\begin{theorem} \label{th4}
Assume $(\mathrm{P_{1})}$-$\mathrm{(P_{2})}$ and let $u\in C^{1,1}_{loc}\cap L_{sp(x,\cdot)}$ such that $u\in (0, 1)$ be a solution of 
$$
(-\Delta)^s_{p(x,\cdot)} u(x)= f(u(x)), \quad x\text{ in } \R^N.
$$
Assume that
\begin{equation} \label{cond.decay2}
f'(t)\leq 0 \quad \text{for } t\leq 1,
\end{equation}
\begin{equation} \label{cond.decay}
\lim_{|x|\to\infty} u(x) =0.
\end{equation}
Then $u$ is radially symmetric around some point in $\R^N$.
\end{theorem}

\begin{proof}
Let $H_{\lambda}$, $u_{\lambda}(x)$ and $w_{\lambda}(x)$ be defined as in Section \ref{sec3}. 

\vspace{0.1cm}
We split the proof in two steps in order to apply the moving planes method.

\medskip

\noindent \textbf{Step 1}. In this step, we show that for $\lambda$ sufficiently negative, it holds
\begin{eqnarray}\label{eq4c}
w_{\lambda}(x)\geq 0, \forall~~x\in H_{\lambda}.
\end{eqnarray}
Due to the decay condition \eqref{cond.decay}  on $u$, there exists $x^0 \in H_\lambda$ such that $w_\lambda(x^0) \min_{H_\lambda} w_\lambda <0$.

\noindent Moreover, using the equation we have that
\begin{equation}\label{eq2c*}
(-\Delta)^{s}_{p(x, \cdot)}u_{\lambda}(x)- (-\Delta)^{s}_{p(x, \cdot)}u(x)
=
f(u_\lambda(x))-f(u(x))=f'(\xi) w_\lambda(x),
\end{equation}
where $\xi_{\lambda}(x)$ is a value between $u_{\lambda}(x)$ and $u(x)$. In particular, we have
$$
u_\lambda(x^0) \leq \xi (x^0) \leq u(x^0).
$$
The decay assumption on $u$ gives that for $\lambda$ negative enough, $u(x^0)$ is small and then $\xi(x^0)$ is small, giving that $f'(\xi(\bar x))\leq 0$ due to \eqref{cond.decay2}. As a consequence, 
$$
(-\Delta)^{s}_{p(x, \cdot)}u_{\lambda}(x)- (-\Delta)^{s}_{p(x, \cdot)}u(x) = f(u_\lambda(x))- f(u(x)) \geq 0.
$$
However, as seen in \eqref{ineqq}, under these conditions we have that
$$
(-\Delta)^{s}_{p(x, \cdot)}u_{\lambda}(x)- (-\Delta)^{s}_{p(x, \cdot)}u(x) = f(u_\lambda(x))- f(u(x)) <0,
$$
which is a contradiction. Therefore $w_\lambda(x)\geq 0$ for all $x\in \Sigma_\lambda$ for $\lambda$ sufficiently negative.
 
\medskip

\noindent \textbf{Step 2}. Finally, if we define the quantity
$$
\lambda_{0}=\sup\{\lambda\leq 0 \text{ such that } w_{\mu}(x)\geq 0,~~ x\in H_{\mu},~~ \mu\leq \lambda\},
$$
using condition \eqref{cond.decay} and \eqref{cond.decay2}  we can proceed analogously as in Step 2 of Theorem \ref{th3} to conclude that $u$ is symmetric about the limiting plane $T_{\lambda_0}$ or $w_{\lambda_0}(x)=0$  for any $x\in H_{\lambda_0}$, which concludes the proof.
\end{proof}

\subsection*{Concluding remarks and open problems:}

We summarize some open problems that arise from our work as follows:
\begin{enumerate}
\item[(i)] Condition $\mathrm{(P_2)}$  plays a key role in the proof of the maximum principle for anti-symmetric functions stated in Theorem
\ref{th2b}.  Note that, Theorem \ref{th2b} is the basic tool in proving  Theorems \ref{le1} and \ref{th3}. We do not have any
knowledge about the proof of Theorem \ref{th2b} without this condition.

\medskip

\item[(ii)] We leave as an open question to find which are the (best) conditions on $f$ and on the decay of $u$ at infinity in order to ensure symmetry of positive solutions $u\in C^{1,1}_{loc}\cap L_{sp(x,\cdot)}$ of
$$
(-\Delta)^s_{p(x,\cdot)} u(x)= f(u(x)) \quad \text{ in } \R^N.
$$
in the case in which $f$ is an increasing function.

\medskip

\item[(iii)] Further interesting research directions would be to address qualitative properties of solutions unbounded domains, for instance
$$
(-\Delta)^s_{p(x,\cdot)} u(x)=f(u(x)) \quad \text{ in } \{x_N>0\} \quad \text{ and } u=0 \text{ in }\{x_N=0\},
$$
or more general unbounded domains such as those given by the epigraph of a Lipschitz function.

\medskip

\item[(iv)] We believe that a valuable research direction is to generalize the abstract approach developed in this paper
to the mixed local and nonlocal case of type as in \cite{val}.
\end{enumerate}

\end{document}